\documentclass[reqno]{amsart}
\usepackage[top=2cm,bottom=2cm,right=2.5cm,left=2.5cm]{geometry}
\usepackage{amssymb}
\usepackage{amsmath, amsthm, amscd, amsfonts, amssymb, graphicx, color}
\usepackage[bookmarksnumbered, colorlinks, plainpages]{hyperref}
\hypersetup{colorlinks=true,linkcolor=red, anchorcolor=green, citecolor=cyan, urlcolor=red, filecolor=magenta, pdftoolbar=true}
 
\usepackage{hyperref}

\newtheorem{theorem}{Theorem}[section]
\newtheorem{lemma}[theorem]{Lemma}

\theoremstyle{definition}
\newtheorem{definition}[theorem]{Definition}
\newtheorem{example}[theorem]{Example}

\theoremstyle{remark}

\numberwithin{equation}{section}

\begin{document}

\title{Best proximity point of generalized
$\theta-\phi-$proximal non-self contractions}

\author[Mohamed Rossafi and Abdelkarim Kari]{Mohamed Rossafi$^{1*}$ and Abdelkarim Kari$^{2}$}

\address{$^{1}$Department of Mathematics, Faculty of Sciences Dhar El Mahraz, University Sidi Mohamed Ben Abdellah, B. P. 1796 Fes Atlas, Morocco}
\email{\textcolor[rgb]{0.00,0.00,0.84}{rossafimohamed@gmail.com}}

\address{$^{2}$Department of Mathematics, Faculty of Sciences Ben M’Sik, Hassan II University, B. P. 7955 Casablanca, Morocco}
\email{\textcolor[rgb]{0.00,0.00,0.84}{abdkrimkariprofes@gmail.com}}

\begin{abstract}
In this manuscript, motivated and inspired by results of Best proximity point of generalized $ F $-proximal non-self contractions, we introduce the concept of generalized $\theta-\phi-$proximal contraction  and prove new best proximity results for these contractions in the setting of a metric
space. Our results generalize and extend many recent results appearing in the literature. An example is being given to demonstrate the usefulness of our results. 
\end{abstract}
\maketitle
\section{Introduction}
It is well known that the Banach contraction theorem is the first outstanding result in the field of the fixed point theory that ensure the existence of unique fixed point in complete metric spaces. Due to its importance, various mathematics steadied many interesting extensions and generalizations \cite{KARI,KARO,RK,ZH}. One of the famous generalizations of the Banach contraction principle \cite{S} for existence of fixed point for self-mapping on metric space  is the theorem by Zheng et al. \cite{ZH} and the contraction introduced by Jleli and  Samet in \cite{JK}.

Best proximity point theorem analyses the condition under which the optimisation problem, namely $ \inf_{x \in A } d(x,Tx)$, has a solution. The point $ x $ is called the best proximity of $ T: A\rightarrow B$, if $ d(x, Tx)=d(A, B)$, where $ \lbrace d(A, B)=\inf d(x,y): x \in A, y \in B \rbrace$. Note that the best proximity point reduces to a fixed point if $ T $ is a self-mapping. Various best proximity point results were established on such spaces \cite{HA,HAY,RK}. 

Sankar Raj \cite{RJ} and Zhang et al. \cite{ZHN} defined the notion of $P-$property and weak $P-$property respectively. Beg et al. \cite{BG} defined the concept of generalized $F$-proximal non-self contractions and obtained some best proximity point theorems for self-mappings.
    
In this paper, inspired by the idea of generalized $F$-proximal non-self contractions, introduced by Beg et al. \cite{BG} in metric spaces, we prove a new existence of best proximity point for generalized $\theta-\phi-$proximal contraction defined on a closed subset of a complete metric space. Our theorems extend, generalize and improve many existing results.
\section{Preliminaries}
Let $(A, B)$ be a pair of non empty subsets of a metric space $(X,d)$. We adopt the following notations:

$d(A,B)=\lbrace\inf d\left(a,b\right):a\in A, b\in B\rbrace$;

$A_0=\lbrace$ $ a\in A $ there exists $b\in A$ such that  $d\left(a,b\right)= d\left(A,B\right)\rbrace$;

$B_0=\lbrace$ $ b\in B $ there exists $a\in A$ such that  $d\left(a,b\right)= d\left(A,B\right)\rbrace$. 
\begin{definition}
\cite{KI} Let $T: A\rightarrow B$ be a mapping. An element $x^{\ast}$ is said to be a best proximity point of $T$ if \begin{equation*}
 d\left(x^{\ast },Tx^{\ast }\right)=d\left(A,B\right).     
\end{equation*}
\end{definition}
\begin{definition}
\cite{RJ} Let $(A,B)$ be a pair of non empty subsets of a metric space $ (X,d) $ such that $ A_0 $ is non empty. Then the pair $(A,B)$ is to have $P$-property if and only if
\begin{equation*}
	\left\lbrace
	\begin{aligned}		    
	d\left(x_1,y_1\right)&=d\left(A,B\right)\\
	d\left(x_2,y_2\right)&=d\left(A,B\right)
	\end{aligned}
	\right.\Rightarrow d(x_1,x_2)=d(y_1,y_2)
	\end{equation*}
 where $x_1, x_2 \in  A_{0}$  and  $y_1, y_2\in  B_{0}$.
\end{definition}
\begin{definition}
\cite{BS} A set $B  $ is called approximately compact with respect to
$ A $ if every sequence $ \lbrace x_{n} \rbrace $ of $ B $ with $  d(y,  x_{n}) \rightarrow d(y,B)$ for some $  y \in A$ has a
convergent subsequence.
\end{definition}
\begin{definition}
\cite{JK} Let  $ \Theta $ be the family of all functions  $\theta$ : $\left]   0,+\infty \right[ $
  $\rightarrow \left] 1 ,+\infty \right[ $ such that
\begin{itemize}
\item[$(\theta_1)$] $\theta$ is strictly increasing;
\item[$(\theta_2)$] For each sequence $x_n\in \left] 0,+\infty \right[$;
\begin{equation*}
\lim_{n\rightarrow 0}x_{n}=0,\,\,\,\text{ if and only if}\,\,\,\lim_{n\rightarrow \infty }\theta\left( x_{n}\right) =1;
\end{equation*}
\item[$(\theta_3)$]  $\theta$ is continuous.
\end{itemize} 
\end{definition}
\begin{definition}
\cite{ZH}  Let   $ \Phi $ be the family of all functions $\phi$:  $\left[ 1,+\infty \right[ $
  $\rightarrow \left[ 1,+\infty \right[ $, such that
 \begin{itemize}
 \item[$(\phi_1)$] $\phi$ is increasing;
\item[$(\phi_2)$] For each  $t\in \left] 1,+\infty \right[ $, $  lim_{n\rightarrow \infty }\phi^{n}( t) =1$;
\item[$(\phi_3)$] $\phi$ is continuous.
\end{itemize}  
\end{definition}
\begin{lemma} \cite{ZH} \label{LZ} If   $\phi $ $\in \Phi$   Then   $ \phi( 1) $=1,
and  $ \phi( t) < t $.
\end{lemma}
\begin{definition}
\cite{ZH}. Let $(X,d)$ be a metric space and $T:X\rightarrow X$ be a mapping.
 
$T$ is said to be a $\theta-\phi-$contraction if there
exist $\theta \in \Theta $ and $\phi \in \Phi $ such that for any
$x,y\in X,$ 
\begin{equation*}
d\left( Tx,Ty\right) >0\Rightarrow \theta \left[ d\left( Tx,Ty\right) \right]
\leq \phi \left [\theta \left(d\left( x,y\right) \right )\right], 
\end{equation*}
\end{definition}
\section{Main result}
In this section, inspired by the notion of $ F $-proximal contraction of the first kind and second kind, we introduce new generalized $\theta-\phi$-proximal first kind and second kind on complete metric space.
\begin{definition}
The mapping 
$ T: A\rightarrow B$ is said to be a generalized $\theta-\phi$-proximal contraction of first kind if there exist  $\theta \in \Theta $,  $\phi \in\Phi $  and $ a,b,c,h\geq 0 $ with $ a+b++2ch $, $ c\neq 1 $ such that
\begin{equation*}
	\left\lbrace
	\begin{aligned}		    
	d\left(u_1,Tv_1\right)&=d\left(A,B\right)\\
	d\left(u_2,Tv_2\right)&=d\left(A,B\right)
	\end{aligned}
	\right.\Rightarrow \theta(d(u_1,u_2))\\
	\leq \phi\left[ \theta\left[  a d\left(v_1,v_2\right)+bd\left(u_1,v_1\right)+cd\left(u_2,v_2\right)+h\left( d\left(v_1,u_2\right)+d\left(v_2,u_1\right)\right)\right]  \right] 
	\end{equation*} 
for all $  u_1, u_2,v_1, v_2\in A$ and $ u_1\neq v_1 $.
\end{definition}
\begin{definition}
The mapping 
$ T:A\rightarrow B$ is said to be a generalized $\theta-\phi $-proximal contraction of second kind if there exist  $\theta \in \Theta $,  $\phi \in\Phi $  and $ a,b,c,h\geq  0 $ with $ a+b++2ch $, $ c\neq 1 $ such that
\begin{equation*}
	\left\lbrace
	\begin{aligned}		    
	d\left(u_1,Tv_1\right)&=d\left(A,B\right)\\
	d\left(u_2,Tv_2\right)&=d\left(A,B\right)
	\end{aligned}
	\right.\Rightarrow \theta(d(Tu_1,Tu_2))\\	
	\end{equation*} 
	\begin{equation*}
	\leq \phi\left[ \theta\left[  a d\left(Tv_1,Tv_2\right)+bd\left(Tu_1,Tv_1\right)+cd\left(Tu_2,Tv_2\right)+h\left( d\left(Tv_1,Tu_2\right)+d\left(Tv_2,Tu_1\right)\right)\right]  \right] 
	\end{equation*}	
for all $  u_1, u_2,v_1, v_2\in A$ and $ Tu_1\neq Tv_1 $.
\end{definition}
\begin{theorem}\label{T1}
Let $ (X,d) $ be a complete metric space and $ (A,B) $ be a pair of non-void closed subsets of $ (X,d) $. If $ B $ is approximately compact with respect to $ A $ and $ T: A\rightarrow B $ satisfy the following conditions :
\begin{itemize}
\item[(i)] $ T\left( A_0\right) \in B_0$ and the pair $\left( A,B\right)$ satisfies the weak $ P $-property;
\item[(ii)] $ T $ is  a generalized $\theta-\phi$-proximal contraction of first kind.
\end{itemize}
Then there exists a unique $ u\in A $ such that $ d(u, Tu)=d(A,B) $. In addition, for any fixed element $ u_{0 }\in A_{0}$, sequence $ \lbrace u_n\rbrace $ defined by 
$$ d( u_{n+1 },T u_{n })= d(A,B),$$
converges to the proximity point.
\end{theorem}
\begin{proof}
Choose an element $ u_{0}\in A_{0}$. As,  $ T\left( A_0\right) \in B_0$,
therefore there is  an element $  u_{1} \in A_{0}$ satisfying 
\begin{equation*}
 d(u_1,T u_0)  = d(A,B).
\end{equation*}
Since $ T(A_0) \in B_0 $, there exists $ u_2\in A_0 $ such that 
$$ d(u_2,Tu_1) =d(A,B).$$
Again, since $ T(A_0) \in B_0 $, there exists $ u_3\in A_0 $ such that  $$ d(u_3,Tu_2) =d(A,B).$$
Continuing this process, by induction, we construct a sequence $ x_n\in A_0 $ such that
\begin{equation*}\label{q}
d\left(u_{n+1},Tu_{n}\right)= d(A,B),  \forall n\in \mathbb{N}. 
\end{equation*}        
Since $ (A,B)$ satisfies the  $ P $ property, we conclude that
\begin{equation}\label{cc}
 d(u_n,u_{n+1})= d(Tu_{n},Tu_{n+1}), \forall n\in \mathbb{N}.
\end{equation}
If $ u_{n_{0}}= u_{n_{0}+1}  $ for some $n_{0}\in\mathbb{N}  $, from (\ref{q}) one obtains
\begin{equation}
 d\left(u_{n_{0}},Tu_{n_{0}}\right)=d\left(u_{n_{0}+1},Tu_{n_{0}}\right)= d(A,B)
\end{equation}
that is, $u_{n_{0}}\in BPP  $. Thus, we suppose that
$ d(u_n,x_{n+1}) >0 $ for all $ n\in\mathbb{N}. $\\
We shall prove that the sequence $ u_n $ is a Cauchy sequence.
Let us first prove that  
\begin{equation*}
\lim_{n\rightarrow \infty }d\left( u_{n},u_{n+1}\right)  =0.
\end{equation*}
As $ T  $ is generalized $(\theta,\phi) $-proximal contraction of the first kind, we have that
\begin{align*}
\theta \left(d\left(u_{n} ,u_{n+1}\right)\right)&\leq \phi\left[ \theta\left[  a d\left(u_{n-1} ,u_{n} \right)+bd\left(u_{n-1} ,u_{n} \right)+cd\left(u_{n} ,u_{n+1} \right)+h\left( d\left(u_{n-1} ,u_{n+1} \right)+d\left(u_{n} ,u_{n} \right)\right)\right]  \right]\\
&=\phi\left[ \theta\left[  a d\left(u_{n-1} ,u_{n} \right)+bd\left(u_{n-1} ,u_{n} \right)+cd\left(u_{n} ,u_{n+1} \right)+h\left( d\left(u_{n-1} ,u_{n+1} \right)\right)\right]  \right]\\
& \leq \phi\left[ \theta\left[  a d\left(u_{n-1} ,u_{n} \right)+bd\left(x_{n-1} ,x_{n} \right)+cd\left(u_{n} ,u_{n+1} \right)+h\left( d\left(u_{n-1} ,u_{n} \right)+d\left(u_{n} ,u_{n+1} \right)\right)\right]  \right]\\
&= \phi\left[ \theta\left[  (a+b+h) d\left(u_{n-1} ,u_{n} \right)+(c+h)d\left(u_{n} ,u_{n+1} \right)\right]  \right]
\end{align*}
Since $  \theta$ is strictly increasing and by Lemma $ \ref{LZ} $, we deduce
\begin{equation*}
d\left(x_{n} ,x_{n+1}\right)<(a+b+h) d\left(x_{n-1} ,x_{n} \right)+(c+h)d\left(x_{n} ,x_{n+1} \right).
\end{equation*}
Thus
\begin{equation*}
d\left(u_{n} ,u_{n+1}\right)<\frac{a+b+h}{1-c-h}( d\left(u_{n-1} ,x_{n} \right)).
\end{equation*}
If $ b+b+c+2h=1 $, we have $ 0< 1-c-h$ and so 
 \begin{equation*}
d\left(u_{n} ,u_{n+1}\right)\leq \frac{a+b+h}{1-c-h}( d\left(u_{n-1} ,u_{n} \right))=d\left(u_{n-1} ,u_{n} \right), \forall n \in\mathbb{N};
\end{equation*}
 Consequently, 
  \begin{equation*}
\theta\left( d\left(u_{n} ,u_{n+1}\right)\right) \leq \phi\left[ \theta \left( d\left(u_{n-1} ,u_{n} \right)\right)\right] 
\end{equation*}
If $ b+b+c+2h<1 $, we have $ 0< 1-c-h$ and so 
 \begin{equation*}
d\left(u_{n} ,u_{n+1}\right)<d\left(u_{n-1} ,u_{n} \right), \forall n \in\mathbb{N};
\end{equation*}
 Consequently, 
  \begin{equation*}
\theta\left( d\left(u_{n} ,u_{n+1}\right)\right) \leq \phi\left[ \theta \left( d\left(u_{n-1} ,u_{n} \right)\right)\right] 
\end{equation*}
It implies
\begin{align*}
\theta \left(d\left(u_{n} ,u_{n+1}\right)\right) 
&\leq \phi\left[ \theta \left(d(x_{n-1},u_{n} \right)\right] \\
&\leq \phi^{2}\left[ \theta\left(d(u_{n-2},u_{n-1} \right)\right] \\
&\leq...\leq \phi ^{n}\left[ \theta \left(d(u_{0},u_{1} \right)\right] .
\end{align*}
Taking the limit as $ n\rightarrow \infty $, we have
\begin{equation*}
1\leq \theta(d\left( u_{n},u_{n+1}\right))\leq \lim_{n\rightarrow \infty }\phi ^{n}\left[ \theta(d\left( u_{0},u_{1}\right))\right]  =1.
\end{equation*}
Since $ \theta\in\Theta $, we obtain
\begin{equation}
\lim_{n\rightarrow \infty }d\left( u_{n},u_{n+1}\right) =0.
\end{equation}
Next, we shall prove that $\left\lbrace  u_{n}\right\rbrace  _{n\in \mathbb{N}}$ is a Cauchy sequence, i.e, $\lim_{n\rightarrow \infty }d\left( u_{n,}u_{m}\right) =0,$ for all $n\in \mathbb{N}$.
Suppose to the contrary that exists $\varepsilon $ $>0$ and sequences $
n_{\left( k\right) }$ and $m_{\left( k\right)}$ of natural numbers such that
\begin{equation}\label{5A}
 m_{\left( k\right) }> n_{\left( k\right) }>k,\ \  d\left( x_{m_{\left( k\right) }},x_{n_{\left( k\right) }}\right) \geq
\varepsilon , \ \ D\left( u_{m_{\left( k\right) -1}},u_{n_{\left( k\right) }}\right) <\varepsilon .
\end{equation} 
Using the triangular inequality, we find that, 
\begin{align}\label{6A}
\varepsilon  \leq d\left( u_{m_{\left( k\right) }},u_{n_{\left( k\right)}}\right) &\leq d\left( u_{m_{\left( k\right) }},u_{n\left( k\right)-1}\right) +d\left( x_{n\left( k\right)-1},x_{n_{\left( k\right)}}\right)\\
& <\varepsilon+d\left( u_{n\left( k\right)-1},u_{n_{\left( k\right)}}\right).
\end{align}
Then, by  $\ref{5A}$ and $\ref{6A}$, it follows that  
\begin{equation}\label{7}
\lim_{k\rightarrow \infty }d\left( u_{m_{\left( k\right) }},u_{n_{\left( k\right)}}\right) =\varepsilon .
\end{equation}
Using the triangular inequality, we find that, 
\begin{align}\label{66A}
\varepsilon  \leq d\left( u_{m_{\left( k\right) }},u_{n_{\left( k\right)}}\right) &\leq d\left( u_{m_{\left( k\right) }},u_{n\left( k\right)+1}\right) +d\left( x_{n\left( k\right)+1},u_{n_{\left( k\right)}}\right)
\end{align}
and
\begin{align}\label{666A}
\varepsilon  \leq d\left( u_{m_{\left( k\right) }},u_{n_{\left( k\right)+1}}\right) &\leq d\left( u_{m_{\left( k\right) }},u_{n\left( k\right)}\right) +d\left( u_{n\left( k\right)},u_{n_{\left( k\right)+1}}\right)
\end{align}
Then, by  $(\ref{66A})$ and $(\ref{666A})$, it follows that  
\begin{equation}\label{77}
\lim_{k\rightarrow \infty }d\left(u_{ m_{\left( k\right)} },u_{n_{\left( k\right)+1}}\right) =\varepsilon .
\end{equation}
Similarly method, we conclude that
\begin{equation}\label{777}
\lim_{k\rightarrow \infty }d\left( u_{m_{\left( k\right) +1}},u_{n_{\left( k\right)}}\right) =\varepsilon .
\end{equation}
Using again the triangular inequality, 
\begin{equation}\label{88}
d\left( u_{m_{\left( k\right) +1}},u_{n_{\left( k\right) +1}}\right) \leq
d\left( x_{m_{\left( k\right) +1}},x_{m_{\left( k\right)} }\right) +d\left(
u_{m\left( k\right) },u_{n_{\left( k\right) }}\right) +d\left( u_{n_{\left(k\right)}},u_{n_{\left( k\right) +1}}\right).
\end{equation}
On the other hand, using triangular inequality, we have 
\begin{equation}\label{99}
d\left( u_{m_{\left( k\right) }},u_{n_{\left( k\right) }}\right) \leq d\left( u_{m_{\left( k\right) }},u_{m_{\left( k\right) +1}}\right) +d\left(u_{m_{\left( k\right) +1}},u_{n_{\left( k\right) +1}}\right) +d\left(u_{n_{\left( k\right) +1}},u_{n_{\left( k\right) }}\right). 
\end{equation}
Letting $k\rightarrow \infty $ in inequality $(\ref{88} )$ and $(\ref{99}) $, we obtain 
\begin{equation}\label{10}
\lim_{k\rightarrow \infty }d\left( u_{m_{\left( k\right) +1}},u_{n_{\left(k\right) +1}}\right) =\varepsilon .
\end{equation}
Substituting $ u_1= x_{m_{\left( k\right)+1}},u_2= x_{n_{\left( k\right)+1}},v_1= u_{m_{\left( k\right)}} $ and $ v_1= u_{n_{\left( k\right)}} $ in  assumption of the theorem, we get
\begin{equation}\label{xx}
	\theta \left( {d\left( u_{m_{\left( k\right) +1}},u_{n_{\left( k\right)+1}}\right) }\right)\leq \phi\left\lbrace \theta\left\lbrace
	\begin{aligned}		    
	&ad\left( u_{m_{\left( k\right)}},u_{n_{\left( k\right)}}\right)\\
	&+bd\left( u_{m_{\left( k\right)}1},u_{n_{\left( k\right)}}\right)\\
	&+cd\left( u_{n_{\left( k\right)}+1},u_{n_{\left( k\right)}}\right)\\
	&+h(d\left( u_{m_{\left( k\right)}},u_{n_{\left( k\right)}+1}\right)+d\left( u_{n_{\left( k\right)}},u_{m_{\left( k\right)}+1}\right))
	\end{aligned}
\right\rbrace \right\rbrace  		
	\end{equation} 
Letting Letting $k\rightarrow \infty $ in (\ref{xx}), and using  $\left( \theta _{1}\right)$, $\left( \theta _{3}\right),$  $\left( \phi _{3}\right)$ and Lemma (\ref{LZ}) we obtain
\begin{equation*}
\theta \left(\varepsilon\right) \leq \phi\left[  \theta\left( a\varepsilon +b\varepsilon+c\varepsilon+2h\varepsilon\right)\right]. 
\end{equation*} 
We derive
\begin{equation*}
\varepsilon <\varepsilon.
\end{equation*}
Which is a contradiction. Thus $ \lim_{n,m\rightarrow \infty }d\left( u_{n},u_{m}\right)= 0 $, which shows that $\lbrace x_{n} \rbrace   $ is a Cauchy sequence. Then there exists $ z\in A $ such that
\begin{equation*}
\lim_{n\rightarrow \infty }d\left( u_{n},u\right)=0 .
\end{equation*}
Also,
\begin{align*}
d\left( u,B\right)&\leq d\left( u,Tu_{n}\right)\\
&\leq d\left( u,x_{n+1}\right)+d\left( u_{n+1},Tu_{n}\right)\\
&=d\left( u,u_{n+1}\right)+d\left( A,B\right)\\
&\leq d\left( u,u_{n+1}\right)+d\left( u,B\right).
\end{align*}
Therefore, $d\left( u,Tu_{n}\right)\rightarrow d\left( u,B\right). $ In spite of the fact that $ B $ is approximately compact with respect to $ A $  , the sequence $ \lbrace Tu_{n}\rbrace $ has a subsequence $ \lbrace Tu_{n_{k}}\rbrace $ converging to some element $ v \in B. $ So it turns out that
 \begin{equation}\label{dd}
d(u,v)= \lim_{n\rightarrow \infty }d\left( u_{n_{k}+1},Tu_{n_{k}}\right)=d(A,B).
\end{equation}
Thus $ u $ must be an element of $ A_{0} $.
Again, since $ T(A_0) \in B_0 $, there exists $ t\in A_0 $ such that  
\begin{equation}
 d(t,Tu) =d(A,B)
\end{equation}
for some element $  t$ in $ A $. Using the weak $ p $-property and (\ref{dd}) we have
$$ d(u_{n_{k}+1}, t) = d(Pu_{n_{k}}, Pu), \forall n_k \in\mathbb{N}.  $$
If for some $ n_0 $, $  d(t,  u_{n_{0}+1}) = 0$, consequently $  d(P u_{n_{0}}, Tu) = 0 $. So $  P u_{n_{0}} =Tu$, hence $  d(A,B) = d(u, Tu)$. Thus the conclusion is immediate. So let for
any $  n \geq 0$, $  d(t,  u_{n+1}) > 0$. Since $ T $ is a generalized $(\theta,\phi)  $-proximal contraction of the first kind, it follows from this that
\begin{equation}\label{gg}
\theta(d(t,u_{n+1}))
	\leq \phi\left[ \theta\left[  a d\left(u,u_{n}\right)+bd\left(t,u\right)+cd\left(u_{n},u_{n+1}\right)+h\left( d\left(u,u_{n+1}\right)+d\left(u_{n},t\right)\right)\right]  \right] 
\end{equation}
Since $ \theta $ and $ \phi $ are two continuous functions, by letting $  n\rightarrow \infty$ in inequality $  (\ref{gg})$, we obtain
\begin{align*}
\theta(d(t,u))&\leq \phi\left[ \theta\left[  (b+h)\left( d\left(u,t\right)\right)\right]  \right]\\
&\leq \phi\left[ \theta\left[  \left( d\left(u,t\right)\right)\right]  \right]\\
&<\theta(d(t,u)).  
\end{align*}
It is a contradiction. Therefore, $ u=t $, that
$$ d(u, Tu) = d(t, Tu) = d(A,B). $$
 Uniqueness: Suppose that there is another best proximity point $ z $ of the mapping $ T $ such that
$$ d(z, Tz)=d(A,B). $$
Since $  T$ is a generalized $(\theta,\phi)  $-proximal contraction of the first kind, it follows from this that
\begin{align*}
\theta(d(z,u))&\leq \phi\left[ \theta\left[  a d\left(z,u\right)+bd\left(z,z\right)+cd\left(u,u\right)+h\left( d\left(z,u\right)+d\left(z,u\right)\right)\right]  \right]\\
&=\phi\left[ \theta\left[  (a+2h) d\left(z,u\right)\right]  \right],   
\end{align*}
which is a contradiction. Thus, $ z $ and $ u $ must be identical. Hence, $  T$ has a unique best proximity point.
\end{proof} 
Next, we state and prove the best proximity point theorem for non-self
generalized $ (\theta,\phi) $-proximal contraction of the second kind.
\begin{theorem}\label{T2}
Let $ (X,d) $ be a complete metric space and $ (A,B) $ be a pair of non-void closed subsets of $ (X,d) $. If $ A $ is approximately compact with respect to $ B $ and $ T: A\rightarrow B $ satisfy the following conditions :
\begin{itemize}
\item[(i)] $ T\left( A_0\right) \in B_0$ and the pair $\left( A,B\right)$ satisfies the weak $ P $-property;
\item[(ii)] $ T $ is continuous generalized $ (\theta,\phi)$-proximal contraction of second kind.
\end{itemize}
Then there exists a unique $ u\in A $ such that $ d(u, Tu)=d(A,B) $
and $  u_n \rightarrow u$, where $  u_0$ is any fixed point in $ A_0 $ and $   d( u_{n+1 },T u_{n })= d(A,B)$ for $  n \geq 0$. Further, if $ z $ is another best proximity point of $ T $, then $  Tu = Tz$.
\end{theorem}
\begin{proof}
Similar to Theorem \ref{T1}, we can find a sequence $ \lbrace u_n\rbrace $ in $ A_0 $ such that
\begin{equation}\label{SS}
d( u_{n+1 },T u_{n })= d(A,B).
\end{equation} 
for all non-negative integral values of $ n $. From the $  p$-property and (\ref{SS}) we get $$  d(u_{n }, u_{n+1 }) =  d(Tu_{n-1 }, Tu_{n }), \forall n \in\mathbb{N}.$$
If for some $ n_{0} $, $ d(u_{n _{0+1 }}, u_{n _{0+2 }})= 0 $, consequently $ d(Tu_{n _{0 }}, Tu_{n _{0 }+1})= 0 $. So
$ Tu_{n _{0 }}= Tu_{n _{0 }+1} $, hence $d(A,B)=d(Tu_{n _{0 }}, T_{n _{0 }+1}).  $ Thus the conclusion is immediate. So let for any $  n \geq 0$,
$  d(Tu_{n }, Tu_{n +1}) > 0$.
We shall prove that the sequence $ u_n $ is a Cauchy sequence.
Let us first prove that  
\begin{equation*}
\lim_{n\rightarrow \infty }d\left( u_{n},u_{n+1}\right)  =0.
\end{equation*} 

As $ T  $ is generalized $(\theta,\phi) $-proximal contraction of the second kind, we have that
\begin{align*}
\theta \left(d\left(Tu_{n} ,Tu_{n+1}\right)\right)&\leq \phi\left[ \theta\left[  a d\left(Tu_{n-1} ,Tu_{n} \right)+bd\left(Tu_{n-1} ,Tu_{n} \right)+cd\left(Tu_{n} ,Tu_{n+1} \right)+h\left( d\left(Tu_{n-1} ,Tu_{n+1} \right)+d\left(Tu_{n} ,Tu_{n} \right)\right)\right]  \right]\\
&=\phi\left[ \theta\left[  a d\left(Tu_{n-1} ,Tu_{n} \right)+bd\left(Tu_{n-1} ,Tu_{n} \right)+cd\left(Tu_{n} ,Tu_{n+1} \right)+h\left( d\left(Tu_{n-1} ,Tu_{n+1} \right)\right)\right]  \right]\\
& \leq \phi\left[ \theta\left[  a d\left(Tu_{n-1} ,Tu_{n} \right)+bd\left(Tu_{n-1} ,Tu_{n} \right)+cd\left(Tu_{n} ,Tu_{n+1} \right)+h\left( d\left(Tu_{n-1} ,Tu_{n} \right)+d\left(Tu_{n} ,Tu_{n+1} \right)\right)\right]  \right]\\
&= \phi\left[ \theta\left[  (a+b+h) d\left(Tu_{n-1} ,Tu_{n} \right)+(c+h)d\left(Tu_{n} ,Tu_{n+1} \right)\right]  \right]
\end{align*}
Since $  \theta$ is strictly increasing and by Lemma $ \ref{LZ} $, we deduce
\begin{equation*}
d\left(Tu_{n} ,Tu_{n+1}\right)<(a+b+h) d\left(Tu_{n-1} ,Tu_{n} \right)+(c+h)d\left(Tu_{n} ,Tu_{n+1} \right).
\end{equation*}
Thus
\begin{equation*}
d\left(Tu_{n} ,Tu_{n+1}\right)<\frac{a+b+h}{1-c-h}( d\left(Tu_{n-1} ,Tu_{n} \right)).
\end{equation*}
If $ b+b+c+2h=1 $, we have $ 0< 1-c-h$ and so 
 \begin{equation*}
d\left(Tu_{n} ,Tu_{n+1}\right)\leq \frac{a+b+h}{1-c-h}( d\left(Tu_{n-1} ,Tu_{n} \right))=d\left(Tu_{n-1} ,Tu_{n} \right), \forall n \in\mathbb{N};
\end{equation*}
 Consequently, 
  \begin{equation*}
\theta\left( d\left(Tu_{n} ,Tu_{n+1}\right)\right) \leq \phi\left[ \theta \left( d\left(Tu_{n-1} ,Tu_{n} \right)\right)\right] 
\end{equation*}
If $ b+b+c+2h<1 $, we have $ 0< 1-c-h$ and so 
 \begin{equation*}
d\left(Tu_{n} ,Tu_{n+1}\right)<d\left(Tu_{n-1} ,Tu_{n} \right), \forall n \in\mathbb{N};
\end{equation*}
 Consequently, 
  \begin{equation*}
\theta\left( d\left(Tu_{n} ,Tu_{n+1}\right)\right) \leq \phi\left[ \theta \left( d\left(Tu_{n-1} ,Tu_{n} \right)\right)\right] 
\end{equation*}
It implies
\begin{align*}
\theta \left(d\left(Tu_{n} ,Tu_{n+1}\right)\right) 
&\leq \phi\left[ \theta \left(d(Tu_{n-1},Tu_{n} \right)\right] \\
&\leq \phi^{2}\left[ \theta\left(d(Tu_{n-2},Tu_{n-1} \right)\right] \\
&\leq...\leq \phi ^{n}\left[ \theta \left(d(Tu_{0},Tu_{1} \right)\right] .
\end{align*}
Taking the limit as $ n\rightarrow \infty $, we have
\begin{equation*}
1\leq \theta(d\left( Tu_{n},Tu_{n+1}\right))\leq \lim_{n\rightarrow \infty }\phi ^{n}\left[ \theta(d\left( Tu_{0},Tu_{1}\right))\right]  =1.
\end{equation*}
Since $ \theta\in\Theta $, we obtain
\begin{equation}
\lim_{n\rightarrow \infty }d\left( Tu_{n},Tu_{n+1}\right) =0.
\end{equation}
Next, we shall prove that $\left\lbrace  Tu_{n}\right\rbrace  _{n\in \mathbb{N}}$ is a Cauchy sequence, i.e, $\lim_{n\rightarrow \infty }d\left( Tu_{n,}Tu_{m}\right) =0,$ for all $n\in \mathbb{N}$.
Suppose to the contrary that exists $\varepsilon $ $>0$ and sequences $
Tn_{\left( k\right) }$ and $Tm_{\left( k\right)}$ of natural numbers such that
\begin{equation}\label{5AA}
 Tm_{\left( k\right) }> Tn_{\left( k\right) }>k,\ \  d\left( Tu_{m_{\left( k\right) }},Tu_{n_{\left( k\right) }}\right) \geq
\varepsilon , \ \ d\left( Tu_{m_{\left( k\right) -1}},Tu_{n_{\left( k\right) }}\right) <\varepsilon .
\end{equation} 
Using the triangular inequality, we find that, 
\begin{align}\label{6A}
\varepsilon  \leq d\left( Tu_{m_{\left( k\right) }},Tu_{n_{\left( k\right)}}\right) &\leq d\left( Tu_{m_{\left( k\right) }},Tx_{n\left( k\right)-1}\right) +d\left( Tu_{n\left( k\right)-1},Tu_{n_{\left( k\right)}}\right)\\
& <\varepsilon+d\left( Tu_{n\left( k\right)-1},Tu_{n_{\left( k\right)}}\right).
\end{align}
Then, by  $\ref{5A}$ and $\ref{6A}$, it follows that  
\begin{equation}\label{7}
\lim_{k\rightarrow \infty }d\left( Tu_{m_{\left( k\right) }},Tu_{n_{\left( k\right)}}\right) =\varepsilon .
\end{equation}
Using the triangular inequality, we find that, 
\begin{align}\label{66A}
\varepsilon  \leq d\left( Tu_{m_{\left( k\right) }},Tu_{n_{\left( k\right)}}\right) &\leq d\left( Tu_{m_{\left( k\right) }},Tu_{n\left( k\right)+1}\right) +d\left(T u_{n\left( k\right)+1},Tu_{n_{\left( k\right)}}\right)
\end{align}
and
\begin{align}\label{666AA}
\varepsilon  \leq d\left( Tu_{m_{\left( k\right) }},Tu_{n_{\left( k\right)+1}}\right) &\leq d\left( Tu_{m_{\left( k\right) }},Tu_{n\left( k\right)}\right) +d\left(T u_{n\left( k\right)},Tu_{n_{\left( k\right)+1}}\right)
\end{align}
Then, by  $(\ref{66A})$ and $(\ref{666A})$, it follows that  
\begin{equation}\label{77}
\lim_{k\rightarrow \infty }d\left(Tu_{ m_{\left( k\right) }},Tu_{n_{\left( k\right)+1}}\right) =\varepsilon .
\end{equation}
Similarly method, we conclude that
\begin{equation}\label{7777}
\lim_{k\rightarrow \infty }d\left( Tu_{m_{\left( k\right) +1}},Tu_{n_{\left( k\right)}}\right) =\varepsilon .
\end{equation}
Using again the triangular inequality, 
\begin{equation}\label{888}
d\left( Tu_{m_{\left( k\right) +1}},Tu_{n_{\left( k\right) +1}}\right) \leq
d\left( u_{m_{\left( k\right) +1}},Tu_{m_{\left( k\right)} }\right) +d\left(
Tu_{m\left( k\right) },Tu_{n_{\left( k\right) }}\right) +d\left( Tu_{n_{\left(k\right)}},Tu_{n_{\left( k\right) +1}}\right).
\end{equation}
On the other hand, using triangular inequality, we have 
\begin{equation}\label{999}
d\left( Tu_{m_{\left( k\right) }},Tu_{n_{\left( k\right) }}\right) \leq d\left( Tu_{m_{\left( k\right) }},Tu_{m_{\left( k\right) +1}}\right) +d\left(Tu_{m_{\left( k\right) +1}},Tu_{n_{\left( k\right) +1}}\right) +d\left(Tu_{n_{\left( k\right) +1}},Tu_{n_{\left( k\right) }}\right). 
\end{equation}
Letting $k\rightarrow \infty $ in inequality $(\ref{888} )$ and $(\ref{999}) $, we obtain 
\begin{equation}\label{100}
\lim_{k\rightarrow \infty }d\left( Tu_{m_{\left( k\right) +1}},Tu_{n_{\left(k\right) +1}}\right) =\varepsilon .
\end{equation}
Substituting $ u_1= Tu_{m_{\left( k\right)+1}},u_2= Tu_{n_{\left( k\right)+1}},v_1= Tu_{m_{\left( k\right)}} $ and $ v_1= Tu_{n_{\left( k\right)}} $ in  assumption of the theorem, we get
\begin{equation}\label{xxX}
	\theta \left( {d\left( Tu_{m_{\left( k\right) +1}},Tu_{n_{\left( k\right)+1}}\right) }\right)\leq \phi\left\lbrace \theta\left\lbrace
	\begin{aligned}		    
	&ad\left( Tu_{m_{\left( k\right)}},Tu_{n_{\left( k\right)}}\right)\\
	&+bd\left( Tu_{m_{\left( k\right)}1},Tu_{n_{\left( k\right)}}\right)\\
	&+cd\left(T u_{n_{\left( k\right)}+1},Tu_{n_{\left( k\right)}}\right)\\
	&+h(d\left( Tu_{m_{\left( k\right)}},Tu_{n_{\left( k\right)}+1}\right)+d\left( Tu_{n_{\left( k\right)}},Tu_{m_{\left( k\right)}+1}\right))
	\end{aligned}
\right\rbrace \right\rbrace  		
	\end{equation} 
Letting Letting $k\rightarrow \infty $ in (\ref{xxX}), and using  $\left( \theta _{1}\right)$, $\left( \theta _{3}\right),$  $\left( \phi _{3}\right)$ and Lemma (\ref{LZ}) we obtain
\begin{equation*}
\theta \left(\varepsilon\right) \leq \phi\left[  \theta\left( a\varepsilon +b\varepsilon+c\varepsilon+2h\varepsilon\right)\right]. 
\end{equation*} 
We derive
\begin{equation*}
\varepsilon <\varepsilon.
\end{equation*}
Which is a contradiction. Thus $ \lim_{n,m\rightarrow \infty }d\left( Tu_{n},Tu_{m}\right)= 0 $, which shows that $\lbrace Tu_{n} \rbrace   $ is a Cauchy sequence. Then there exists $ v\in B $ such that
\begin{equation*}
\lim_{n\rightarrow \infty }d\left( Tu_{n},v\right)=0 .
\end{equation*}
Also,
\begin{align*}
d\left( v,A\right)&\leq d\left( v,Tu_{n}\right)\\
&\leq d\left( v,u_{n+1}\right)+d\left( u_{n+1},Tu_{n}\right)\\
&=d\left( v,u_{n+1}\right)+d\left( A,B\right)\\
&\leq d\left( v,u_{n+1}\right)+d\left( v,A\right).
\end{align*}
Therefore, $d\left( v,Tu_{n}\right)\rightarrow d\left( v,A\right). $ Since $ A $ is approximately compact with respect to $ B $  , the sequence $ \lbrace u_{n}\rbrace $ has a subsequence $ \lbrace u_{n_{k}}\rbrace $ converging to some element $ u \in A. $ So it turns out that
 \begin{equation}\label{dd}
d(u,v)= \lim_{n\rightarrow \infty }d\left( u_{n_{k}+1},Tu_{n_{k}}\right)=d(A,B).
\end{equation}
Because $ T $ is a continuous mapping,
$$  d(u, Tu) = \lim_{n\rightarrow \infty } d(u_{n+1}, Tu_{n}) = d(A,B).$$
 Uniqueness: Suppose that there is another best proximity point $ z $ of the mapping $ T $ such that
$$ d(z, Tz)=d(A,B). $$
Since $  T$ is a generalized $(\theta,\phi)  $-proximal contraction of the first second, it follows from this that
\begin{align*}
\theta(d(Tz,Tu))&\leq \phi\left[ \theta\left[  a d\left(Tz,Tu\right)+bd\left(Tz,Tz\right)+cd\left(Tu,Tu\right)+h\left( d\left(Tz,Tu\right)+d\left(Tz,Tu\right)\right)\right]  \right]\\
&=\phi\left[ \theta\left[  (a+2h) d\left(Tz,Tu\right)\right]  \right],   
\end{align*}
which is a contradiction. Thus, $ z $ and $ u $ must be identical. Hence, $  T$ has a unique best proximity point.

\begin{theorem}\label{T3}
Let $ (X,d) $ be a complete metric space and $ (A,B) $ be a pair of non-void closed subsets of $ (X,d) $. Let $ T: A\rightarrow B $ satisfy the following conditions :
\begin{itemize}
\item[(i)] $ T\left( A_0\right) \in B_0$ and the pair $\left( A,B\right)$ satisfies the weak $ P $-property;
\item[(ii)]  $ T $ is a generalized $ (\theta,\phi)$-proximal contraction of the first kind as well as a generalized $ (\theta,\phi)$-proximal contraction of the second kind.
\end{itemize}
Then there exists a unique $ u\in A $ such that $ d(u, Tu)=d(A,B) $
and $  u_n \rightarrow u$, where $  u_0$ is any fixed point in $ A_0 $ and $   d( u_{n+1 },T u_{n })= d(A,B)$ for $  n \geq 0$.
\end{theorem}
\begin{proof}
Similar to Theorem \ref{T1}, we find a sequence $  \lbrace u_{n } \rbrace $ in $ A_0 $ such that
$$ d( u_{n+1 }, T u_{n }) = d(A,B) $$
for all non-negative integral values of $ n $. Similar to Theorem \ref{T1}, we can show
that sequence $  \lbrace u_{n } \rbrace $ is a Cauchy sequence. Thus converges to some element $ u $ in $  A$. As in Theorem \ref{T2}, it can be shown that the sequence $  \lbrace Tu_{n } \rbrace $ is a Cauchy sequence and converges to some element $ v $ in $  B$. Therefore,
\begin{equation}\label{BB}
d(u, v) = \lim_{n\rightarrow \infty }d( u_{n+1 }, T u_{n }) = d(A,B).
\end{equation}
Eventually, u becomes an element of $  A_0$. In light of the fact that $  T(A_0)\in B_0$,
$$  d(t, Tu) = d(A,B)$$
for some element $t  $ in $  A$. From the $  p$-property framework and (\ref{BB},) we get
$$ d( u_{n+1 }, t) = d(T u_{n }, Tu), \forall n \in \mathbb{N}. $$
If for some $ n_0 $, $  d(t,  u_{n_{ 0}+1}) = 0$, consequently $  d(Tu_{n_{ 0}}, Tu) = 0$ . So $  Tu_{n_{ 0}} =Tu$, hence $  d(A,B) = d(u, Tu)$. Thus the conclusion is immediate. So let for any $  n \geq 0$, $  d(t, u_{n+1 }) > 0$. Since $  T$ is a generalized $  (\theta,\phi)$-proximal contraction
of the first kind, it can be seen that
\begin{align}\label{gggg}
\theta(d(t, u_{n+1}))&\leq \phi\left[ \theta \left(  
ad(u,  u_{n }) + bd(t, u) + cd( u_{n },  u_{n+1 })
+ h[d(u,  u_{n+1 }) + d( u_{n }, t)\right)\right]  .
\end{align}
Since $ \theta $ and $ \phi $ are two continuous functions, by letting $  n\rightarrow \infty$ in inequality $  (\ref{gggg})$, we obtain, 
$  d(u, Tu) = d(t, Tu) = d(A,B).$
Also, as in the theorem \ref{T1}, the uniqueness of the best proximity point of mapping $ T $ follows.
\end{proof}
\end{proof}
\begin{example}
Let $ X=\lbrace \lambda_{n }: n\in\mathbb{N} \rbrace $  with the
metric $ d(x, y) = \vert x-y \vert $  for all $  x, y \in X $, where the sequence $ G_{n }  $, defined by
\begin{align*}
&\lambda_{1 } = 1\\
&\lambda_{2 } = 1 + 2\\
&\lambda_{3 } = 1 + 2 + 3\\
&...\\
&\lambda_{n } = 1 + 2 + 3 + . . . + n.
\end{align*}
We know, $  (X, d)$ is a complete metric space.
Let $  A = G_{3n } : n\in\mathbb{N}$ and $  B = G_{3n-1 } : n\in\mathbb{N}$. It is easy to see that $  d(A,B) = 3$,
$ A_0 = A $ and $ B_0 =B $. Define a mappings $  T : A \rightarrow B$, by $T(\lambda_{3n }) = \lambda_{3n-1 }$ for all $n \geq 1$.
It is clear that $  A$ is approximately compact with respect to $ B $, $  (A,B)$ satisfies the $ p $-property, $  T$ is continuous and $T(A_0) \subseteq B_0$ . We
will show that $ T $ is an $ (\theta,\phi) $-proximal contraction with $ \theta\in\Theta $ and $ \phi\in\Phi $ that is $ \theta(t)=e^{t} $ and $ \phi(t)=t^{\frac{1}{2}} $. Observe that,
With out of generality, we may assume that $  n < m$, and since
\begin{align*}
&\lambda_{3n-1 } = 1 + 2 + 3 + . . . + 3n - 1,\\
&\lambda_{3m-1 } =  1 + 2 + 3 + . . . + 3m - 1,\\
&\lambda_{3n } = 1 + 2 + 3 + . . . + 3n - 1+3n,\\
&\lambda_{3m } =  1 + 2 + 3 + . . . + 3m - 1+3m.
\end{align*}
It follow that,
\begin{align*}
d(T(\lambda_{3n }), T(\lambda_{3m}))& = \vert \lambda_{3n-1 }- \lambda_{3m-1}\vert\\
&= 3n + (3n + 1) + . . . + (3m - 1),
\end{align*}
\begin{align*}
d(\lambda_{3n }, \lambda_{2m})& = \vert \lambda_{2n }- \lambda_{2m}\vert\\
&= 3n + (3n + 1) + . . . + (3m ),
\end{align*}
and
\begin{align*}
d(T(\lambda_{2n }), T(\lambda_{3m}))-d(\lambda_{3n }, \lambda_{3m})& = \vert \lambda_{3n-1 }- \lambda_{3m-1}\vert -\vert \lambda_{3n }- \lambda_{3m}\vert\\
&= 3n -3m.
\end{align*}
So that,
\begin{align*}
e^{d(T(\lambda_{3n }), T(\lambda_{3m}))-d(\lambda_{3n }, \lambda_{3m}) })&=
\frac{e^{ d(T(\lambda_{3n }), T(\lambda_{3m}))   }}{e^{ d(\lambda_{3n }, \lambda_{3m})   }}\\
&=e^{ 3n -3m})\\
&=e^{ -3(m -n)})\\
&\leq e^{ -3}= \frac{1}{e^{ 3}}.
\end{align*}
So that,
\begin{align*}
e^{ d(T(\lambda_{3n }), T(\lambda_{3m})) }+1&=\theta(d(T(\lambda_{3n }), T(\lambda_{3m})))\\
& \leq e^{ d(\lambda_{3n }, \lambda_{3m})   }\frac{1}{e^{ 3}}+1\\
&\leq \frac{e^{ d(\lambda_{3n }, \lambda_{3m})   }+2}{2}\\
&=\phi\left[ \theta(d(\lambda_{3n }, \lambda_{3m}))\right] .
\end{align*}
Consequently, $ T $ is an generalized $(\theta,\phi)  $-proximal contraction of the second kind with $ a=1 $, $  b = c = h = 0$. Thus, all the conditions of
Theorem \ref{T2} are satisfied. Hence, $ T $ has a unique best proximity point  and there exist $  \lambda_ { 3}\in A$ such that
$$  d(\lambda_ { 3}, T\lambda_ { 3}) = d(\lambda_ { 3}, \lambda_ { 2}) = 3 = d(A,B)$$
\end{example}
\bibliographystyle{amsplain}

\begin{thebibliography}{99}
\bibitem{HAY} H. Aydi, H. Lakzian, Z. Mitrović, S. Radenović,
Best Proximity Points of MT-Cyclic Contractions with Property UC, Numerical Functional Analysis and Optimization; Volume 41, 2020.
	  
\bibitem {S} S. Banach, Sur les opérations dans les ensembles abstraits et leur application aux équations intégrales. Fundam. Math. 3(1), 133–181 (1922).

 \bibitem {BS} S. S. Basha, P. Veeramani, Best proximity pair theorems for multifunctions with open fibres. J. Approx. Theory 103, 119–129 (2000).
    
\bibitem {BG} I. Beg, G. Mani, A. J. Gnanaprakasam, Best proximity point of generalized F-proximal non-self contractions
 Journal of Fixed Point Theory and Applications 23 (4), 1-11

\bibitem{KI} A. Eldred, W. Kirk, P. Veeramani, Proximal normal structure and relatively nonexpansive mappings, Stud. Math. 171(3), 283-293 (2005).

\bibitem{JK} M. Jleli, B. Samet, A new generalization of the Banach contraction principle, J. Inequal. Appl., 2014 (2014), 8 pages. 
	
\bibitem{KARI} A. Kari, M. Rossafi, E. Marhrani, M. Aamri, Fixed point theorem for Nonlinear $F-$contraction via $w-$distance,  Adv. Math. Phys, 2020(2020), Article ID 6617517.

\bibitem{KARO} A. Kari, M. Rossafi, E. Marhrani, M. Aamri, $\theta-\phi-$contraction on $\left(\alpha,\eta \right)-$complete rectangular $b-$metric spaces, Int. J. Math. Mathematical Sciences, (2020), Article ID 5689458.

\bibitem{HA} V. Parvaneh, M. Reza Haddadi and H. Aydi, On Best Proximity Point Results for Some Type of Mappings. Hindawi Journal of Function Spaces Volume 2020, Article ID 6298138, 6 pages.
	
\bibitem{RJ} V. S. Raj, A best proximity point theorem for weakly contractive non-self-mappings. Nonlinear Anal. 2011, 74, 4804–4808.
	
\bibitem{RK} M.  Rossafi, A. Kari, Some fixed point theorems for $ F $-expansive mapping in generalized metric spaces, Open J. Math. Anal. 2021, 5(2), 17-30.

\bibitem{RK} M. Rossafi, A. Kari, Best Proximity Point Theorems for $ \alpha $-Proximal $ \theta,\phi $-non-self mappings, Asian Journal of Mathematics and Applications, 2021.
	
\bibitem{ZHN} J. Zhang, Y. Su, Q. Cheng, A note on ‘A best proximity point theorem for Geraghty-contractions’. Fixed Point
Theory Appl. 2013, 99.

\bibitem{ZH} D. Zheng , Z. Cai , P. Wang,  New fixed point theorems for ($\theta-\phi  $)-contraction in complete metric. spaces. Journal of Non linear Sciences and Applications.  2017;10(5):2662, and Information Processing 6 (2008), 433-446.
\end{thebibliography}

\end{document}